\documentclass[11pt]{article}

\usepackage[a4paper,left=20mm,top=20mm,right=20mm,bottom=25mm]{geometry}

\usepackage{amsmath, amsfonts, amssymb, amsthm}
\usepackage[noend]{algorithmic}
\usepackage[linesnumbered,ruled,vlined]{algorithm2e}

\usepackage{mathtools}
\usepackage[mathscr]{euscript}
\usepackage{wasysym}

\usepackage{graphicx}
\usepackage{xcolor}
\usepackage{enumitem}
\usepackage{authblk}
\usepackage{underscore}

\usepackage{color}

\usepackage[colorlinks=true,urlcolor=black,linkcolor=blue,citecolor=blue]{hyperref}
\usepackage[font=footnotesize,width=.85\textwidth,labelfont=bf]{caption}

\newtheorem{theorem}{Theorem}
\newtheorem{lemma}{Lemma}
\newtheorem{corollary}{Corollary}
\newtheorem{definition}{Definition}
\newtheorem{proposition}{Proposition}

\usepackage{bigfoot}
\DeclareNewFootnote[para]{default}
\DeclareNewFootnote{A}[arabic]
\DeclareNewFootnote[para]{B}[fnsymbol]
\makeatletter
\def\moverlay{\mathpalette\mov@rlay}
\def\mov@rlay#1#2{\leavevmode\vtop{%
    \baselineskip\z@skip \lineskiplimit-\maxdimen
    \ialign{\hfil$\m@th#1##$\hfil\cr#2\crcr}}}
\newcommand{\charfusion}[3][\mathord]{
  #1{\ifx#1\mathop\vphantom{#2}\fi
    \mathpalette\mov@rlay{#2\cr#3}
  }
  \ifx#1\mathop\expandafter\displaylimits\fi}
\DeclareRobustCommand\bigop[1]{%
  \mathop{\vphantom{\sum}\mathpalette\bigop@{#1}}\slimits@
}
\newcommand{\bigop@}[2]{%
  \vcenter{%
    \sbox\z@{$#1\sum$}%
    \hbox{\resizebox{\ifx#1\displaystyle.9\fi\dimexpr\ht\z@+\dp\z@}{!}{$\m@th#2$}}%
  }%
}
\makeatother

\newcommand{\cupdot}{\charfusion[\mathbin]{\cup}{\cdot}}

\newcommand{\lca}{\operatorname{lca}}
\newcommand{\parent}{\operatorname{parent}}
\newcommand{\child}{\operatorname{child}}
\newcommand{\HH}{\mathcal{H}}
\newcommand{\LL}{L^{(2)}}

\newcommand{\SPEC}{\newmoon}
\newcommand{\HHGT}{\triangle}
\newcommand{\DUPL}{\square}
\newcommand{\OTHER}{\HHGT}
\newcommand{\EHGT}{\mathbb{I}}
\newcommand{\AX}[1]{\textnormal{#1}}

\providecommand{\keywords}[1]{\textbf{\textit{Keywords: }} #1}

\title{Combining Orthology and Xenology Data in a\\Common Phylogenetic Tree}

\author[1]{Marc Hellmuth}
\author[2]{Mira Michel}
\author[3]{Nikolai N. N{\o}jgaard}
\author[4,5]{David Schaller}
\author[4-8]{Peter F.\ Stadler}

\affil[1]{Department of Mathematics, Faculty of Science,
  Stockholm University, SE-10691 Stockholm, Sweden
  \texttt{marc.hellmuth@math.su.se}}

\affil[2]{Faculty of Mathematics and Computer Science,
  Fernuniversit{\"a}t Hagen, Universit{\"a}tsstrasse 47, D-58097 Hagen, Germany
  \texttt{mira.michel@studium.fernuni-hagen.de}}

\affil[3]{Department of Mathematics and Computer Science, University of Southern
  Denmark, Odense M, Denmark
  \texttt{nnoej10@gmail.com}}

\affil[4]{Max Planck Institute for Mathematics in the
  Sciences, Leipzig, Germany 
  \texttt{sdavid@bioinf.uni-leipzig.de}}

\affil[5]{Bioinformatics Group, Department of Computer Science, and
  Interdisciplinary Center for Bioinformatics, Universit{\"a}t Leipzig,
  H{\"a}rtelstrasse 16-18, D-04107 Leipzig, Germany
  \texttt{studla@bioinf.uni-leipzig.de}}

\affil[6]{Institute for Theoretical Chemistry,
  University of Vienna, Vienna, Austria}

\affil[7]{Facultad de Ciencias, Universidad Nacional de Colombia, Bogot{\'a},
  Colombia}

\affil[8]{Santa Fe Institute, Santa Fe, New Mexico, USA}

\date{\ }

\begin{document}
  
\maketitle 

\abstract{
  A rooted tree $T$ with vertex labels $t(v)$ and set-valued edge labels
  $\lambda(e)$ defines maps $\delta$ and $\varepsilon$ on the pairs of
  leaves of $T$ by setting $\delta(x,y)=q$ if the last common ancestor
  $\lca(x,y)$ of $x$ and $y$ is labeled $q$, and $m\in \varepsilon(x,y)$ if
  $m\in\lambda(e)$ for at least one edge $e$ along the path from
  $\lca(x,y)$ to $y$. We show that a pair of maps $(\delta,\varepsilon)$
  derives from a tree $(T,t,\lambda)$ if and only if there exists a common
  refinement of the (unique) least-resolved vertex labeled tree
  $(T_{\delta},t_{\delta})$ that explains $\delta$ and the (unique) least
  resolved edge labeled tree $(T_{\varepsilon},\lambda_{\varepsilon})$ that
  explains $\varepsilon$ (provided both trees exist). This result remains
  true if certain combinations of labels at incident vertices and edges are
  forbidden.
}

\bigskip
\noindent

\sloppy

\section{Introduction} 

An important task in evolutionary biology and genome research is to
disentangle the mutual relationships of related genes. The evolution of a
gene family can be understood as a tree $T$ whose leaves are genes and
whose inner vertices correspond to evolutionary events, in particular
speciations (where genomes are propagated into different lineages that
henceforth evolve independently), duplications (of genes within the same
genome) and horizontal gene transfer (where copies of an individual genes
are transferred into an unrelated species) \cite{Fitch:00}. Mathematically,
these concepts are described in terms of rooted trees $T$ with vertex
labels $t$ representing event types and edge labels $\lambda$
distinguishing vertical and horizontal inheritance. On the other hand,
orthology (descent from a speciation) or xenology (if the common history
involves horizontal transfer events) can be regarded as binary relation on
the set $L$ of genes. Given the orthology or xenology relationships, one
then asks whether there exists a vertex or edge labeled tree $T$ with leaf
set $L$ that ``explains'' the relations
\cite{Hellmuth:13a,Geiss:18a}. Here, we ask when such relational orthology
and xenology data are consistent. A conceptually similar question is
addressed in a very different formal setting in \cite{Jones:17}.

Instead of considering a single binary orthology or xenology relation, we
consider here multiple relations of each type. This is more conveniently
formalized in terms of maps that assign finite sets of labels.  Two types
of maps are of interest: Symbolic ultrametrics, i.e., symmetric maps
determined by a label at the last common ancestor of two genes
\cite{Boecker:98} generalize orthology. Fitch maps, i.e., non-symmetric
maps determined by the union of labels along the path connecting two genes
\cite{Hellmuth:20c}, form a generalization of xenology. For both types of
maps unique least-resolved trees (minimal under edge-contraction) exist and
can be constructed by polynomial time algorithms
\cite{Boecker:98,Hellmuth:20c}. Here we consider the problem of finding
trees that are simultaneously edge- and vertex-labeled and simultaneously
explain both types of maps. We derive a simple condition for the existence
of explaining trees and show that there is again a unique least-resolved
tree among them. We then consider a restricted version of problem motivated
by concepts of observability introduced in \cite{Nojgaard:18a}.

\section{Preliminaries} 

\subsection{Trees and Hierarchies} 

Let $T$ be a rooted tree with vertex set $V(T)$, leaf set
$L(T)\subseteq V(T)$, set of inner vertices
$V^0(T)\coloneqq V(T)\setminus L(T)$, root $\rho\in V^0(T)$, and edge set
$E(T)$. An edge $e=\{u,v\}\in E(T)$ is an \emph{inner} edge if
$u,v\in V^0(T)$.  The ancestor partial order on $V(T)$ is defined by
$x\preceq_T y$ whenever $y$ lies along the unique path connecting $x$ and
the root.  We write $x \prec_T y$ if $x\preceq_T y$ and $x\ne y$.  For
$v\in V(T)$, we set
$\child(v)\coloneqq\{u\mid\{v,u\}\in E(T),\, u\prec_T v\}$.  All trees $T$
considered here are \emph{phylogenetic}, i.e., they satisfy
$|\child(v)|\ge 2$ for all $v\in V^0(T)$. The \emph{last common ancestor}
of a vertex set $W\subseteq V(T)$ is the unique $\preceq_T$-minimal vertex
$\lca_T(W)\in V(T)$ satisfying $w\preceq_T\lca_T(W)$ for all $w\in W$. For
brevity, we write $\lca_T(x,y)\coloneqq\lca_T(\{x,y\})$. Furthermore, we
will sometimes write $vu\in E(T)$ as a shorthand for ``$\{u,v\}\in E(T)$
with $u\prec_T v$.'' We denote by $T(u)$ the subtree of $T$ rooted in $u$
and write $L(T(u))$ for its leaf set.

Furthermore, $L^T_v\coloneqq \{(x,y)\mid x,y\in L(T), \lca_T(x,y)=v\}$
denotes the set of pairs of leaves that have $v$ as their last common
ancestor. By construction, $L^T_v\cap L^T_{v'}=\emptyset$ if $v\ne
v'$. Since $T$ is phylogenetic, we have $L^T_v\ne\emptyset$ for all
$v\in V^0(T)$, i.e., $\mathcal{L}(T)\coloneqq \{L^T_v\mid v\in V^0(T)\}$ is
a partition of the set of distinct pairs of vertices.

A hierarchy on $L$ is set system $\HH\subseteq 2^L$ such that (i)
$L\in\HH$, (ii) $A\cap B\in\{A,B,\emptyset\}$ for all $A,B\in\HH$, and
(iii) $\{x\}\in\HH$ for all $x\in L$. There is a well-known bijection
between rooted phylogenetic trees $T$ with leaf set $L$ and hierarchies on
$L$, see e.g.\ \cite[Thm.\ 3.5.2]{Semple:03}. It is given by
$\HH(T) \coloneqq \{ L(T(u)) \mid u\in V(T) \}$; conversely, the tree
$T_{\HH}$ corresponding to a hierarchy $\HH$ is the Hasse diagram w.r.t.\
set inclusion. Thus, if $v=\lca_T(A)$ for some $A\subseteq L(T)$, then
$L(T(v))$ is the inclusion-minimal cluster in $\HH(T)$ that contains $A$
\cite{Hellmuth:21q}.

Let $T$ and $T^*$ be phylogenetic trees with $L(T)=L(T^*)$. We say that
$T^*$ is a \emph{refinement} of $T$ if $T$ can be obtained from $T^*$ by
contracting a subset of inner edges or equivalently if and only if
$\HH(T)\subseteq\HH(T^*)$.

\begin{lemma}
  Let $T^*$ be a refinement of $T$ and $u^*v^*\in E(T^*)$. Then there is a
  unique vertex $w\in V(T)$ such that $L(T(w))\in \HH(T)$ is
  inclusion-minimal in $\HH(T)$ with the property that
  $L(T^*(v^*))\subsetneq L(T(w))$. In particular, if $\lca_{T^*}(x,y)=u^*$,
  then $\lca_T(x,y)=w$.
  \label{lem:technical}
\end{lemma}
\begin{proof}
  Let $u^*v^*\in E(T^*)$. Since $\HH(T)\subseteq\HH(T^*)$,
  $L(T)=L(T^*)\in \HH(T)$ and $v^*$ is not the root of $T^*$, there is a
  unique inclusion-minimal $A\in\HH(T)$ with $L(T^*(v^*))\subsetneq A$,
  which corresponds to a unique vertex $w\in V(T)$ that satisfies
  $L(T(w))=A$.
  In the following, we denote with $w^*\in V(T^*)$ the unique vertex that
  satisfies $A=L(T^*(w^*))$, which exists since
  $A\in \HH(T)\subseteq\HH(T^*)$.  Now let $x,y\in L(T)$ be two leaves with
  $\lca_{T^*}(x,y)=u^*$.  From $v^*\prec_{T^*} u^*$, we obtain
  $L(T^*(v^*))\subsetneq L(T^*(u^*))$ and
  $L(T^*(u^*))\subseteq L(T^*(w^*))=L(T(w))$. Hence, we have
  $L(T^*(u^*))\subseteq L(T(w))$, which implies $x,y\in L(T(w))$ and thus
  also $z\coloneqq \lca_T(x,y)\preceq_T w$. Denote by $z^*\in V(T^*)$ the
  unique vertex in $T^*$ with $L(T^*(z^*))=L(T(z))$.  Since $z\preceq_T w$,
  it satisfies $L(T^*(z^*))\subseteq L(T^*(w^*))$. Since
  $x,y\in L(T^*(z^*))\cap L(T^*(u^*))\neq \emptyset$, we either have
  $L(T^*(u^*))\subseteq L(T^*(z^*))$ or
  $L(T^*(z^*))\subsetneq L(T^*(u^*))$. In the second case, we obtain
  $\lca_{T^*}(x,y)\preceq_{T^*} z^*\prec_{T^*} u^*$, a contradiction to
  $\lca_{T^*}(x,y)=u^*$. In the first case, we have
  $L(T^*(v^*))\subsetneq L(T^*(u^*))\subseteq L(T(z))\subseteq L(T(w))$.
  Due to inclusion minimality of $L(T(w))$ we have
  $L(T(z)) = L(T(w))$. Thus $\lca_{T}(x,y) = z = w$. 
\end{proof}
Lemma~\ref{lem:technical} ensures that, for every $u^*\in V^0(T^*)$, there
is a unique $w\in V(T)$ such that $\lca_T(x,y)=w$ for all
$(x,y)\in L_{u^*}^{T^*}$, and thus $L_{u^*}^{T^*}\subseteq L_w^T$. Thus we
have 
\begin{corollary}
  \label{cor:technical}
  If $T^*$ is a refinement of $T$, then the partition $\mathcal{L}(T^*)$ is
  a refinement of $\mathcal{L}(T)$.
\end{corollary}

\subsection{Symbolic Ultrametrics}

We write $\LL \coloneqq \{(x,y)\mid x,y\in L,\, x\ne y\}$ for the
``off-diagonal'' pairs of leafs and let $M$ be a finite set.
\begin{definition} 
  A tree $T$ with leaf set $L$ and labeling $t:V^0(T)\to M$ of its inner
  vertices \emph{explains} a map $\delta:\LL\to M$ if
  $t(\lca(x,y)) = \delta(x,y)$ for all distinct $x,y\in L$.
\end{definition}
Such a map must be symmetric since $\lca_T(x,y)=\lca_T(y,x)$ for all
$x,y\in L$. A shown in \cite{Boecker:98}, a map $\delta:\LL\to M$ can be
explained by a labeled tree $(T,t)$ if and only if $\delta$ is a
\emph{symbolic ultrametric}, i.e., iff, for all pairwise distinct
$u,v,x,y\in L$ holds (i) $\delta(x,y)=\delta(y,x)$ (symmetry), (ii)
$\delta(x,y)=\delta(y,u)=\delta(u,v)\ne\delta(y,v)=\delta(x,v)=\delta(x,u)$
is never satisfied (co-graph property), and (iii)
$|\{\delta(u,v),\delta(u,x),\delta(v,x)\}|\le 2$ (exclusion of rainbow
triangles). In this case, there exists a unique least-resolved tree
$(T_{\delta},t_{\delta})$ (that explains $\delta$) with a discriminating
vertex labeling $t_{\delta}$, i.e., $t_{\delta}(x)\ne t_{\delta}(y)$ for
all $ xy\in E(T_{\delta})$ \cite{Boecker:98,Hellmuth:13a}. This tree
$(T_{\delta},t_{\delta})$ is also called a discriminating representation of
$\delta$ \cite{Boecker:98}.

The construction of symbolic ultrametrics could also be extended to maps
$\tilde\delta:\LL\to 2^M$, i.e, to allow multiple labels at each
vertex. However, this does not introduce anything new. To see this, we note
that the sets of vertex pairs $L^T_v$ that share the same last common
ancestor are pairwise disjoint. In particular, $\tilde\delta$ thus must be a
fixed element in $2^M$ on each $L^T_v$, $v\in V^0$, and thus we think of
the images $\tilde\delta(x,y)$ simply as single labels ``associated to''
elements in $2^M$ rather than sets of labels.

\begin{lemma} \label{lem:orthoclu}
  Let $\delta:\LL\to M$ be a symbolic ultrametric with least-resolved tree
  $(T_{\delta},t_{\delta})$. Then there is a map $t:V(T)\to M$ such that
  $(T,t)$ explains $\delta$ if and only if $T$ is a refinement of
  $T_{\delta}$. In this case, the map $t$ is uniquely determined by $T$ and
  $\delta$.
\end{lemma}
\begin{proof}
  Suppose $(T,t)$ explains $\delta$ and let $e=vu\in E(T)$ be an edge with
  $\delta(u)=\delta(v)$ and $u\prec v$.  Note that both $u$ and $v$ must be
  inner vertices.  Let $T/e$ denote the tree obtained from $T$ by
  contracting the edge $e$, i.e., removing $e$ from $T$ and identifying $u$
  and $v$.  We will keep the vertex $v$ in $T/e$ as placeholder for the
  identified vertices $u$ and $v$.  By construction, $T/e$ has the clusters
  $\HH(T/e)=\HH(T)\setminus\{L(T(u))\}$. Set $t_{T/e}(x)=t(x)$ for all
  $x\in V^0(T)\setminus\{u\}$.  Clearly, $v$ is the unique vertex in $T/e$
  such that $L((T/e)(v))$ is inclusion-minimal with property
  $L(T(u'))\subsetneq L((T/e)(v))$ for any $u'\child_{T}(u)$. Therefore, by
  Lemma~\ref{lem:technical}, $\lca_{T}(x,y)=u$ implies $\lca_{T/e}(x,y)=v$,
  and thus, we have $t(\lca_T(x,y))=t_{T/e}(\lca_{T/e}(x,y))$ for all
  $(x,y)\in\LL$, and thus $(T/e,t_{T/e})$ explains $\delta$. Stepwise
  contraction of all edges whose endpoints have the same label eventually
  results in a tree $T'$ and a map $t'$ such that $t'(x)\ne t'(y)$ for all
  edges of $T'$. Thus $(T',t')$ coincides with the unique discriminating
  representation of $\delta$, i.e., $(T',t')=(T_{\delta},t_{\delta})$. By
  construction, $T$ is a refinement of $T_{\delta}$.

  Conversely, let $\delta$ be a symbolic ultrametric with (unique)
  discriminating representation $(T_{\delta},t_{\delta})$ and let $T$ be a
  refinement of $T_{\delta}$. By Cor.~\ref{cor:technical}, $\mathcal{L}(T)$
  is a refinement $\mathcal{L}(T_{\delta})$. Hence, the map $t:V^0(T)\to M$
  specified by $t(\lca_T(x,y))\coloneqq t_{\delta}(\lca_{T_{\delta}}(x,y))$
  for all $(x,y)\in \LL$ is well-defined. By construction, therefore,
  $(T,t)$ explains $\delta$. In particular, therefore, every refinement $T$
  of $T_{\delta}$ admits a vertex labeling $t$ such that $(T,t)$ explains
  $\delta$.  The choice of $t$ is unique since every inner vertex of a
  phylogenetic tree is the last common ancestor of at least one pair of
  vertices, and thus no relabeling of an inner vertex preserves the
  property that the resulting tree explains $\delta$. 
\end{proof}

\subsection{Fitch Maps}
  
\begin{definition} 
  A tree $T$ with edge labeling $\lambda:E(T)\to 2^N$, with finite $N$,
  explains a map $\varepsilon:\LL\to 2^N$ if for all $k\in N$ holds:
  $k\in\varepsilon(x,y)$ iff $k\in\lambda(e)$ for some edge along the
  unique path in $T$ that connects $\lca_T(x,y)$ and $y$.
\end{definition}
A map $\varepsilon:\LL\to 2^N$ that is explained by a tree $(T,\lambda)$ in
this manner is a \emph{Fitch map} \cite{Hellmuth:20c}. A Fitch map is
called \emph{monochromatic} if $|N|=1$.  Like symbolic ultrametrics, Fitch
maps are explained by unique least resolved trees. The key construction is
provided by the sets
$U_{\neg m}[y]\coloneqq \{x\in L\setminus\{y\}\mid
m\notin\varepsilon(x,y)\} \cup\{y\}$ for $y\in L$ and $m\in N$. Let us
write
$\mathcal{N}_{\varepsilon}\coloneqq\{ U_{\neg m}[y] \mid y\in L,\, m\in
N\}$. Then $\varepsilon$ is a Fitch map if and only if (i)
$\mathcal{N}_{\varepsilon}$ is hierarchy-like, i.e.,
$A\cap B\in\{A,B,\emptyset\}$ for all $A,B\in\mathcal{N}_{\varepsilon}$ and
(ii) $|U_{\neg m}[y']|\le |U_{\neg m}[y]|$ for all $y\in L$, $m\in N$, and
$y'\in U_{\neg m}[y]$ \cite[Thm.~3.11]{Hellmuth:20c}.

Fitch maps allow some freedom in distributing labels on the edge set. The
precise notion of ``least-resolved'' thus refers to the fact that it is
neither possible to contract edges nor to remove subsets of labels from an
edge. The unique least-resolved tree for a Fitch map $\varepsilon$, called
the $\varepsilon$-tree $(T_{\varepsilon},\lambda_{\varepsilon})$, is
determined by the hierarchy
$\HH(T_{\varepsilon})=\mathcal{N}_{\varepsilon}\cup\{L\}\cup
\big\{\{x\}\mid x\in L\big\}$ and the labeling
$\lambda_{\varepsilon}(\parent(v),v)\coloneqq\{m\in N \mid \exists y\in L
\text{ s.t.\ } L(T_{\varepsilon}(v))=U_{\neg m}[y]\}$ for all
$e=\{\parent(v),v\}\in E(T_{\varepsilon})$ \cite[Thm.~4.4]{Hellmuth:20c}.

Let $(T,\lambda)$ and $(T',\lambda')$ be two edge-labeled trees on the same
leaf set and with $\lambda:E(T)\to 2^N$ and $\lambda':E(T')\to 2^N$. Then
$(T,\lambda)$ is a refinement of $(T',\lambda')$, in symbols
$(T',\lambda')\le(T, \lambda)$ if (i) $\HH(T')\subseteq \HH(T)$ and (ii) if
$L(T(v))=L(T'(v'))$, then
$\lambda'(\parent_{T'}(v'),v')\subseteq \lambda(\parent_{T}(v),v)$.
\begin{proposition} 
  \label{prop:coarse}
  \cite[Prop.4.3, Thm.4.4]{Hellmuth:20c}\quad If $(T,\lambda)$ explains
  $\varepsilon$, then
  $(T_{\varepsilon},\lambda_{\varepsilon})\le(T,\lambda)$. Furthermore,
  $(T_{\varepsilon},\lambda_{\varepsilon})$ is the unique least-resolved
  tree that explains $\varepsilon$.  In particular,
  $(T_{\varepsilon},\lambda_{\varepsilon})$ minimizes
  $\ell_{\min}\coloneqq \sum_{e\in E(T_{\varepsilon})}
  |\lambda_{\varepsilon}(e)|$.
\end{proposition}

\begin{lemma}\label{lem:fitchclu}
  Let $\varepsilon:\LL\to 2^N$ be a Fitch map with least-resolved tree
  $(T_{\varepsilon},\lambda_{\varepsilon})$.  Then there exists an edge
  labeling $\lambda:E(T)\to 2^N$ such that $(T,\lambda)$ explains
  $\varepsilon$ if and only if $T$ is a refinement of $T_{\varepsilon}$.
\end{lemma}
\begin{proof}
  Suppose $(T,\lambda)$ explains $\varepsilon$. By Prop.~\ref{prop:coarse},
  this implies $(T_{\varepsilon},\lambda_{\varepsilon})\le (T,\lambda)$,
  i.e., $T$ is a refinement of $T_{\varepsilon}$.  Conversely, let
  $\varepsilon$ be a Fitch map with least-resolved tree
  $(T_{\varepsilon},\lambda_{\varepsilon})$ and let $T$ be a refinement of
  $T_{\varepsilon}$. Define, for all edges $\{\parent_{T}(v),v\}\in E(T)$,
  the edge labeling
  \begin{equation}
    \label{eq:lambda}
    \lambda(\{\parent_{T}(v),v\}) \coloneqq
    \begin{cases} 
      \lambda_{\varepsilon}(\parent_{T_{\varepsilon}}(v'),v')
      & \text{ if } L(T(v))=L(T_{\varepsilon}(v')), \\
      \emptyset & \text{ otherwise.}
    \end{cases}
  \end{equation}
  The map $\lambda$ is well-defined, since there is at most one
  $v'\in V(T_{\varepsilon})$ with $L(T(v))=L(T_{\varepsilon}(v'))$.
  
  \textit{Claim.} $(T,\lambda)$ and
  $(T_{\varepsilon},\lambda_{\varepsilon})$
  explain the same Fitch map $\varepsilon$.\\
  By assumption, $(T_{\varepsilon},\lambda_{\varepsilon})$ explains
  $\varepsilon$.  Let $(a,b)\in\LL$, $k\in N$, and let $\varepsilon'$ be
  the Fitch map explained by $(T,\lambda)$.  First, suppose
  $k\in \varepsilon(a,b)$, i.e., there is an edge
  $e'=\{\parent_{T_{\varepsilon}}(w'),w'\}$ with
  $k\in\lambda_{\varepsilon}(e')$ such that
  $w'\prec_{T_{\varepsilon}}\lca_{T_{\varepsilon}}(a,b)$ by the definition
  of Fitch maps.  We have $a\notin L(T_{\varepsilon}(w'))$.  Since $T$ is a
  refinement of $T_{\varepsilon}$, there is a vertex $w\in V(T)$ with
  $L(T(w))=L(T_{\varepsilon}(w'))$.  In particular, therefore,
  $\lambda(\{\parent_{T}(w),w\})=\lambda_{\varepsilon}(e')$.  This together
  with the fact that $a\notin L(T_{\varepsilon}(w'))=L(T(w))$ immediately
  implies $k\in \varepsilon'(a,b)$.  Now suppose $k\in \varepsilon'(a,b)$.
  Hence, there is an edge $e=\{\parent_{T}(v),v\}$ with
  $v\prec_{T}\lca_{T}(a,b)$ and $k\in \lambda(e)$.  By construction of
  $\lambda$, the latter implies that there is a vertex
  $v'\in V(T_{\varepsilon})$ with $L(T(v))=L(T_{\varepsilon}(v'))$ and, in
  particular,
  $k\in\lambda_{\varepsilon}(\parent_{T_{\varepsilon}}(v'),v')$.  The
  latter together with $a\notin L(T(v))=L(T_{\varepsilon}(v'))$ implies
  that $k\in \varepsilon(a,b)$.  Since $(a,b)\in\LL$ and $k\in N$ were
  chosen arbitrarily, we conclude that $\varepsilon=\varepsilon'$, and
  thus, $(T,\lambda)$ also explains $\varepsilon$.  
\end{proof}
  
The labeling $\lambda$ defined in Eq.(\ref{eq:lambda}) satisfies
$\ell_{\min} = \sum_{e\in T(e)} |\lambda(e)|$ by construction and
Prop.~\ref{prop:coarse}. Furthermore, we observe that $(T^*,\lambda^*)$ is
obtained from $(T,\lambda)$ by contracting only edges with
$\lambda(e)=\emptyset$. More precisely, $e$ is contracted if and only if
$e$ is an inner edge with $\lambda(e)=\emptyset$. This implies
\begin{corollary}
  \label{cor:fitchclu}
  Suppose $(T,\lambda')$ explains the Fitch map $\varepsilon$. Then
  $\lambda: E(T)\to 2^N$ given by Eq.~\eqref{eq:lambda} is the unique
  labeling such that $(T,\lambda)$ explains $\varepsilon$ and
  $\sum_{e\in E(T)}|\lambda(e)|=\ell_{\min}$.
\end{corollary}
\begin{proof}
  Suppose $(T,\lambda'')$ explains $\varepsilon$ and
  $\sum_{e\in E(T)}|\lambda''(e)|=\ell_{\min}$. By Prop.~\ref{prop:coarse},
  we have $(T_{\varepsilon},\lambda_{\varepsilon})\le (T,\lambda'')$ and
  thus
  $\lambda_{\varepsilon}(\parent_{T_{\varepsilon}}(v'),v')\subseteq
  \lambda''(\parent_{T}(v),v)$ if $L(T_{\varepsilon}(v'))=L(T(v))$.  Since,
  moreover,
  $\lambda_{\varepsilon}(\parent_{T_{\varepsilon}}(v'),v')=
  \lambda(\parent_{T}(v),v)$ if $L(T_{\varepsilon}(v'))=L(T(v))$ by
  Eq.(\ref{eq:lambda}), minimality of $\lambda''$ implies
  $\lambda''=\lambda$. 
\end{proof}

\section{Tree-like Pairs of Maps}

Symbolic ultrametrics and Fitch maps on $\LL$ derive from trees in very
different ways by implicitly leveraging information about inner vertices
and edges of the \textit{a priori} unknown tree. It is of interest,
therefore, to know when they are consistent in the sense that they can be
simultaneously explained by a tree.
\begin{definition}
  An ordered pair $(\delta,\varepsilon)$ of maps $\delta:\LL\to M$ and
  $\varepsilon:\LL\to 2^N$ is \emph{tree-like} if there is a tree $T$
  endowed with a vertex labeling $t:V^0(T)\to M$ and edge labeling
  $\lambda:\LL\to 2^N$ such that $(T,t)$ explains $\delta$ and
  $(T,\lambda)$ explains $\varepsilon$.
\end{definition}

Naturally, we ask when $(\delta,\varepsilon)$ is explained by a vertex and
edge labeled tree $(T,t,\lambda)$, i.e., when $(\delta,\varepsilon)$ is a
tree-like pair of maps on $\LL$. Furthermore, we ask whether a tree-like
pair of maps is again explained by a unique least-resolved tree
$(T^*,t^*,\lambda^*)$. 

\begin{theorem}
  \label{thm:OOXX-tree}
  Let $\delta:\LL\to M$ and $\varepsilon:\LL\to 2^N$. Then
  $(\delta,\varepsilon)$ is tree-like if and only if 
  \begin{enumerate}
  \item $\delta$ is a symbolic ultrametric. 
  \item $\varepsilon$ is a Fitch map.
  \item $\HH^*\coloneqq \HH(T_{\delta})\cup   \HH(T_{\varepsilon})$ is a
    hierarchy.
  \end{enumerate}
  In this case, there is a unique least-resolved vertex and edge labeled
  tree $(T^*,t^*,\lambda^*)$ explaining $(\delta,\varepsilon)$. The tree
  $T^*$ is determined by $\HH(T^*)=\HH^*$, the vertex labeling $t^*$ is
  uniquely determined by $t_{\delta}$ and the edge labeling $\lambda^*$
  with minimum value of $\sum_{e\in E(T^*)} |\lambda^*(e)|$ is uniquely
  determined by $\lambda_{\varepsilon}$.
\end{theorem}
\begin{proof}
  Suppose $(\delta,\varepsilon)$ is tree-like, i.e., there is a tree
  $(T,t,\lambda)$ such that $(T,t)$ explains $\delta$ and $(T,\lambda)$
  explains $\varepsilon$. Thus $\delta$ is a symbolic ultrametric and
  $\varepsilon$ is a Fitch map. Furthermore, $T$ is a refinement of 
  least-resolved trees $T_{\delta}$ and $T_{\varepsilon}$ because of the
  uniqueness of these least-resolved trees, and we have
  $\HH(T_{\delta})\subseteq \HH(T)$ and $\HH(T_{\varepsilon})\subseteq \HH(T)$
  and thus $\HH^*\subseteq \HH(T)$. Since $\HH(T)$ is a hierarchy and the
  subset $\HH^*$ contains both $L$ and all singletons $\{x\}$ with $x\in L$,
  $\HH^*$ is a hierarchy. 

  Conversely, suppose conditions (1), (2), and (3) are satisfied. The first
  two conditions guarantee the existence of the least-resolved tree
  $(T_{\delta},t_{\delta})$ and $(T_{\varepsilon},\lambda_{\varepsilon})$
  explaining $\delta$ and $\varepsilon$, respectively. Thus
  $\HH^*=\HH(T_{\delta})\cup\HH(T_{\varepsilon})$ is
  well-defined. Condition (3) stipulates that $\HH^*$ is a hierarchy and
  thus there is a unique tree $T^*$ such that $\HH(T^*)=H^*$, which by
  construction is a refinement of both $T_{\delta}$ and
  $T_{\varepsilon}$. By Lemmas~\ref{lem:orthoclu} and~\ref{lem:fitchclu},
  $T^*$ can be equipped with a vertex-labeling $t^*$ and an edge-labeling
  $\lambda^*$ such that $(T^*,t^*)$ explains $\delta$ and $(T^*,\lambda^*)$
  explains $\varepsilon$, respectively. Thus $(\delta,\varepsilon)$ is
  tree-like.

  We now show that $(T^*,t^*,\lambda^*)$ is least-resolved w.r.t.\
  $(\delta,\varepsilon)$ and thus that for every $e\in E(T^*)$, the tree
  $T'\coloneqq T^*/e$ does not admit a vertex labeling $t':V^0(T')\to M$
  and an edge-labeling $\lambda':E(T')\to 2^N$ such that $(T',t',\lambda')$
  explains $(\delta,\varepsilon)$.  Let $e= \{\parent(v),v\}\in E(T^*)$.
  Hence, $L(T^*(v)) \in \HH(T^*)$.  If $v\in L(T^*)$, then we have
  $L(T^*)\ne L(T')$ and the claim trivially holds. Thus suppose that
  $v\in V^0(T)$ in the following.  Since the edge $e$ is contracted in
  $T'$, we have $\HH(T') = \HH(T^*)\setminus \{L(T(v))\}$ and thus,
  $\HH(T_\delta)\not\subseteq \HH(T')$ or
  $\HH(T_{\varepsilon})\not\subseteq \HH(T')$. Thus $T'$ is not a
  refinement of $T_{\delta}$ or $T_{\varepsilon}$.  By
  Lemma~\ref{lem:orthoclu} and~\ref{lem:fitchclu}, respectively, this
  implies that there is no $t'$ such that $(T',t')$ explains $\delta$ or no
  $\lambda'$ such that $(T',\lambda')$ explains $\varepsilon$,
  respectively.  Thus $(T^*,t^*,\lambda^*)$ is least-resolved w.r.t.\
  $(\delta,\varepsilon)$.

  It remains to show that $(T^*,t^*,\lambda^*)$ is unique.  Since $T^*$ is
  uniquely determined by $\HH^*$, it suffices to show that the labeling of
  $T^*$ is unique. This, however, follows immediately from
  Lemma~\ref{lem:orthoclu} and Cor.~\ref{cor:fitchclu}, respectively.  
\end{proof}

We note that every refinement $T$ of the least-resolved tree
$(T^*,t^*,\lambda^*)$ admits a vertex labeling $t:V^0(T)\to M$ and an edge
labeling $\lambda: E(T)\to 2^N$ such that $(T,t,\lambda)$ explains
$(\delta,\varepsilon)$. 

\begin{theorem}
  \label{thm:perf}
  Given two maps $\delta:\LL\to M$ and $\varepsilon:\LL\to 2^N$ it can be
  decided in $O(|L|^2 |N|)$ whether $(\delta,\varepsilon)$ is tree-like. In
  the positive case, the unique least-resolved tree $(T^*,t^*,\lambda^*)$
  can be obtained with the same effort.
\end{theorem}
\begin{proof} 
  Based on Theorem \ref{thm:OOXX-tree}, a possible algorithm consists of
  three steps: (i) check whether $\delta$ is a symbolic ultrametric, (ii)
  check whether $\varepsilon$ is a Fitch map and, if both statements are
  true, (iii) compute
  $\HH^*\coloneqq \HH(T_{\delta})\cup \HH(T_{\varepsilon})$ and use this
  information to compute the unique least-resolved vertex and edge labeled
  tree $(T^*,t^*,\lambda^*)$.  By \cite[Thm.\ 6.2]{Hellmuth:20c}, the
  decision whether $\varepsilon$ is a Fitch map and, in the positive case,
  the construction of the least-resolved tree
  $(T_{\varepsilon},\lambda_{\varepsilon})$ can be achieved in
  $O(|L|^2 |N|)$ time.  Moreover, it can be verified in $O(|L|^2)$ whether
  or not a given map $\delta$ is a symbolic ultrametric, and, in the
  positive case, the discriminating tree $(T_{\delta},t_{\delta})$ can be
  computed within the same time complexity (cf.\ \cite[Thm.~7]{HSW:17}).
  The common refinement $T$ with
    $\HH(T)=\HH(T_{\delta})\cup\HH(T_{\varepsilon})$ can be computed in
    $O(|L|)$ time using \texttt{LinCR} \cite{Schaller:21v}.
    
  The edge labels $\lambda^*$ are then carried over from
  $(T_{\varepsilon},\lambda_{\varepsilon})$ using the correspondence
  between $u^*v^*\in E(T^*)$ and $uv\in E(T_{\varepsilon})$ iff
  $L(T^*(v^*))=L(T_\varepsilon(v))$, otherwise
  $\lambda^*(\{u^*,v^*\})=\emptyset$. This requires $O(|L|\cdot |N|)$
  operations.  The vertex labels can then be assigned by computing, for all
  $(x,y)\in\LL$, the vertex $v=\lca_{T^*}(x,y)$ and assigning
  $t^*(v)=\delta(x,y)$ in quadratic time using a fast last common ancestor
  algorithm \cite{Harel:84}. Thus we arrive at a total performance bounds
  of $O(|L|^2|N|))$.  
\end{proof}

\section{Tree-like Pairs of Maps with Constraints}

One interpretation of tree-like pairs of maps $(\delta,\varepsilon)$ is to
consider $\delta$ as the orthology relation and $\varepsilon$ as the
xenology relation. In such a setting, certain vertex labels $t(v)$ preclude
some edge labels $\lambda(\{v,u\})$ with $u\prec v$. For example, a
speciation vertex cannot be the source of a horizontal transfer edge.  We
use the conventional notations $t(u)=\SPEC$ and $t(v)=\DUPL$ for speciation
and duplication vertices \cite{Geiss:20b}, respectively, set $t(u)=\OTHER$
for a third vertex type, and consider the monochromatic Fitch map
$\varepsilon\colon\LL\to\{\emptyset,\EHGT\}$.  Thus, we require that
$\lambda(\{v,u\})=\EHGT$ and $u\prec_T v$ implies $t(v)=\OTHER$
\cite{Nojgaard:18a,Tofigh:11,Bansal:12}.  This condition simply states that
neither a speciation nor a gene duplication is the source of a horizontal
transfer.

In \cite{Nojgaard:18a}, we considered evolutionary scenarios that satisfy
another rather stringent observability condition:
\begin{description}
  \item[\AX{(C)}] For every $v\in V^0(T)$, there is a child $u\in\child(v)$
  such that $\lambda(\{v,u\})=\emptyset$.
\end{description}
We call a Fitch map $\lambda$ that satisfies \AX{(C)} a \emph{type-C Fitch
  map}.  In this case, for every $v\in V^0(T)$, there is a leaf
$x\in L(T(v))$ such that $\lambda(e)=\emptyset$ for all edges along the
path from $v$ to $x$.  As an immediate consequence of \AX{(C)}, we observe
that, given $|L|\ge 2$, for every $x\in L$ there is a $y\ne x$ such that
$\varepsilon(x,y)=\emptyset$. This condition is not sufficient, however, as
the following example shows. Consider the tree $((x,y),(x',y'))$ in Newick
notation, with the edges in the two cherries $(x,y)$ and $(x',y')$ being
labeled with $\emptyset$, and two $\EHGT$-labeled edges incident to the
root. Then, for every $z\in L$, we have $\varepsilon(z,z')=\emptyset$,
where $z'$ is the sibling of $z$, but condition \AX{(C)} is not satisfied.
In a somewhat more general setting, we formalize these two types of
labeling constraints as follows:
\begin{definition}
  Let $\delta:\LL\to M$ and $\varepsilon:\LL\to 2^N$ be two maps and
  $M_{\emptyset}\subseteq M$.  Then, $(\delta,\varepsilon)$ is
  $M_{\emptyset}$-tree-like if there is a tree $(T,t,\lambda)$ that explains
  $(\delta,\varepsilon)$ and the labeling maps $t:V^0(T)\to M$ and
  $\lambda:E(T)\to 2^N$ satisfy \AX{(C)} and
  \begin{description}
  \item[\AX{(C1)}] If $t(v)\in M_{\emptyset}$, then $\lambda(\{v,u\})=\emptyset$
    for all $u\in\child(v)$.
  \end{description}
\end{definition}
Hence, $M_{\emptyset}$ puts extra constraints to the vertex and edge labels
on trees that satisfy \AX{(C)} and explain $(\delta,\varepsilon)$. Note, an
$\emptyset$-tree-like ($M_{\emptyset} = \emptyset$) must only satisfy
\AX{(C)} and \AX{(C1)} can be omitted.

\begin{theorem}
  \label{thm:CC}
  Let $\delta:\LL\to M$ and $\varepsilon:\LL\to 2^N$ be two maps and
  $M_{\emptyset}\subseteq M$.  Then, $(\delta,\varepsilon)$ is
  $M_{\emptyset}$-tree-like if and only if $(\delta,\varepsilon)$ is
  tree-like and its least-resolved tree $(T^*,t^*,\lambda^*)$ satisfies
  \AX{(C)} and \AX{(C1)}.  
\end{theorem}
\begin{proof}
  If $(\delta,\varepsilon)$ is tree-like and its least-resolved tree
  $(T^*,t^*,\lambda^*)$ satisfies \AX{(C)} and \AX{(C1)},  then
  $(\delta,\varepsilon)$ is $M_{\emptyset}$-tree-like by definition. For
  the converse, suppose $(\delta,\varepsilon)$ is $M_{\emptyset}$-tree-like and
  let  $(T,t,\lambda)$ be a vertex and edge labeled tree that explains
  $(\delta,\varepsilon)$ and satisfies \AX{(C)} and \AX{(C1)}. 

  Let $\lambda'$ be the edge labeling for $T$ as specified in
  Eq.(\ref{eq:lambda}) where $(T_{\varepsilon},\lambda_{\varepsilon})$ is
  replaced by $(T^*,\lambda^*)$.  By the arguments in the proof of
  Lemma~\ref{lem:fitchclu}, $(T,\lambda')$ still explains $\varepsilon$ and
  hence, $(T,t,\lambda')$ explains $(\delta,\varepsilon)$. Moreover, since
  $\ell_{\min}\coloneqq \sum_{e\in E(T^*)} |\lambda^*(e)|$ and by
  construction of $\lambda'$, we have
  $\ell_{\min} = \sum_{e\in E(T)} |\lambda'(e)|$.  Since
  $(T^*,\lambda^*)\leq(T,\lambda')$, it must hold
  $\lambda_{\varepsilon}(e')\subseteq\lambda'(e)$ for all
  $e'=\parent(v')v'\in E(T_{\varepsilon})$ and $e=\parent(v)v\in E(T)$ with
  $L(T(v)) = L(T_{\varepsilon}(v'))$. Since $\lambda'$ is minimal by
  construction, we have $\lambda_{\varepsilon}(e') = \lambda'(e)$ for all
  corresponding edges $e$ and $e'$.  In particular, it must hold that
  $\lambda(e)=\emptyset$ implies $\lambda'(e) = \emptyset$ for all
  $e\in E(T)$. To see this, assume for contradiction there is some edge
  $e=uv\in E(T)$ with $\lambda(e)=\emptyset$ but
  $\lambda'(e) \neq \emptyset$.  Since $(T,\lambda)$ satisfies \AX{(C)},
  there is a path from $u$ to some leaf $y\in L(T)$ that consists of edges
  $f$ with label $\lambda(f) = \emptyset$ only and that contains the edge
  $e$.  Hence, for $x\in L(T(u))\setminus L(T(v))$, we have
  $\lca_T(x,y) = u $ and thus, $\varepsilon(x,y)=\emptyset$. However, since
  we assume that $\lambda'(e)=N'\neq\emptyset$, we obtain
  $N'\subseteq \varepsilon(x,y)\neq \emptyset$; a contradiction.  Now it is
  easy to verify that $(T,t,\lambda')$ still satisfies \AX{(C)} and
  \AX{(C1)}.

  \begin{figure}[t]
    \begin{center}
      \includegraphics[width=0.65\textwidth]{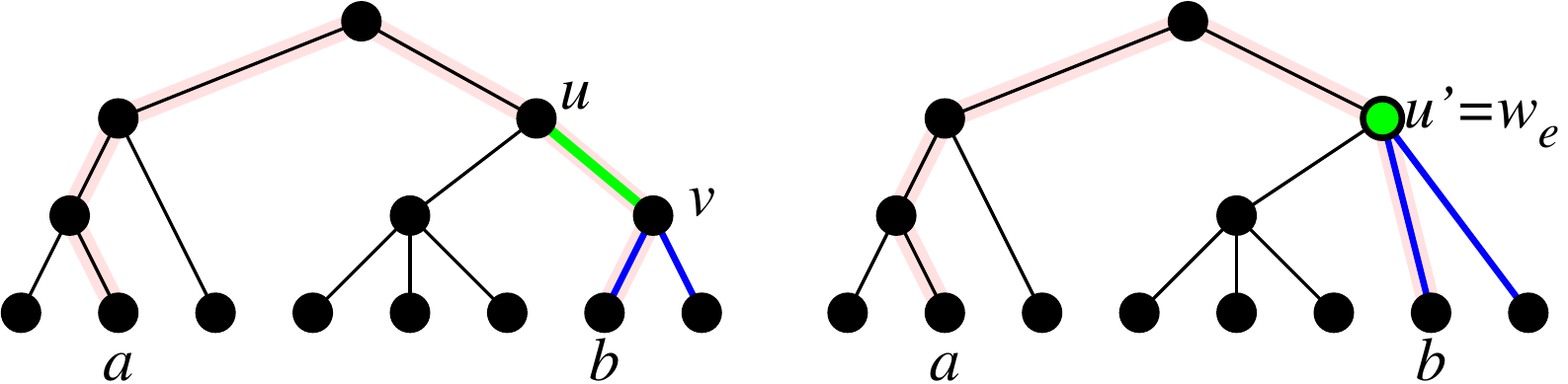}
    \end{center}
    \caption{Effect of an edge contraction on paths in $T$. All paths
      traversing the contracted edge $e=uv$ in $T$ correspond to paths in
      $T/e$ in which $e$ is contracted. All other path remain
      unchanged. Furthermore $w_e=u'$, i.e., the edge contraction
      corresponds to the deletion of $L(T(v))$ from $\HH(T)$.}
    \label{fig:contract}
  \end{figure}

  Now consider edge contractions, Fig.~\ref{fig:contract}. To obtain $T^*$
  we are only allowed to contract edges $e = uv\in E(T)$ that satisfy
  $t(u)=t(v)$ and $\lambda'(e)=\emptyset$. The latter follows from the fact
  that edges $uv$ with $t(u)\neq t(v)$ cannot be contracted without losing
  the information of at least one of the labels $t(u)$ or $t(v)$ and
  minimality of $\lambda'$, since otherwise the labels $\lambda'(e)$ do not
  contribute to the explanation of the Fitch map and thus would have been
  removed in the construction of $\lambda'$. For such an edge $e$, the tree
  $(T/e,t_{T/e},\lambda'_{T/e})$ is obtained by contracting the edge $e=uv$
  to a new vertex $w_e$ and assigning $t_{T/e}(w_e)=t(v)=t(u)$ and
  keeping the edge labels of all remaining edges. The tree
  $(T/e,t_{T/e},\lambda'_{T/e})$ then explains $(\delta,\varepsilon)$.  To
  see this, we write $y\coloneqq\lca_T(a,b)$ and
  $y'\coloneqq\lca_{T/e}(a,b)$ for distinct $a,b\in L$ and compare for
  $c\in\{a,b\}$ the path $P_{yc}$ in $T$ and $P'_{y'c}$ in $T/e$.  If $y=u$
  or $y=v$ then $y'=w_e$. The paths therefore either consist only of
  corresponding edges, in which case the edge labels are the same, or they
  differ exactly by the contraction of $e$. The latter does not affect the
  explanation of $\varepsilon(a,b)$ because $\lambda'(e)=\emptyset$. Since
  $t(u)=t(v)$, contraction of $uv$ also does not affect $\delta$.
    
  In particular, therefore, neither $u$ nor $v$ is a leaf, i.e., $e$ is an
  inner edge.  Condition \AX{(C)} is trivially preserved under contraction
  of inner edges.  Suppose $t(v)=t(u)\in M_{\emptyset}$ and thus
  $t_{T/e}(w_e)\in M_{\emptyset}$.  Since $(T,t,\lambda')$ satisfies
  \AX{(C1)} we have $\lambda'(\{v,u'\})=\lambda'(\{u,u''\}) = \emptyset$ for
  all $u'\in\child(v)$ and all $u''\in \child(u)$ and thus after
  contracting $e$ it holds that $\lambda'_{T/e}(w_e,w')=\emptyset$ for all
  $w'\in\child_{T/e}(w_e)=\child_{T}(v)\cupdot\child_{T}(u)$. Otherwise,
  $t(u)=t(v)\notin M_{\emptyset}$ and thus by construction
  $t_{T/e}(w_e)\notin M_{\emptyset}$. In summary, $(T/e,t',\lambda)$
  satisfies \AX{(C)} and \AX{(C1)}. Repeating this coarse graining until no
  further contractible inner edges are available results in the unique
  least-resolved tree $(T^*,t^*,\lambda^*)$.
\end{proof} 
Since the unique least-resolved tree $(T^*,t^*,\lambda^*)$ can be computed
in quadratic time by Thm.~\ref{thm:perf}, and it suffices by
Thm.~\ref{thm:CC} to check \AX{(C)} and \AX{(C1)} for
$(T^*,t^*,\lambda^*)$, the same performance bound applies to the
recognition of constrained tree-like pairs of maps.

We note that an analogous result holds if only \AX{(C)} or only \AX{(C1)}
is required for $(T,t,\lambda)$. Furthermore, one can extend \AX{(C1)} in
such a way that for a set $\mathcal{Q}$ of pairs $(q,m)$ with $q\in M$ and
$m\in N$ of labels that are incompatible at a vertex $v$ and an edge $vv'$
with $v'\in\child(v)$. The proof of Thm.~\ref{thm:CC} still remains valid
since also in this case no forbidden combinations of vertex an edge colors
can arise from contracting an edge $e=uv$ with $t(u)=t(v)$. In the special
case $\delta(x,y)=1\notin M_{\emptyset}$ for all $(x,y)\in \LL$, one
obtains $t^*(u)=1$ for all $u\in V(T^*)$ and thus
$(T^*,\lambda^*)=(T_{\varepsilon},\lambda_{\varepsilon})$ and \AX{(C1)}
imposes no constraint. Hence, Thm.~\ref{thm:CC} specializes to
\begin{corollary}
  A Fitch map $\varepsilon$ is type-C if and only if its least-resolved
  tree $(T_{\varepsilon},\lambda_{\varepsilon})$ satisfies \AX{(C)}.
\end{corollary}

In \cite{Nojgaard:18a} a stronger version of condition \AX{(C)} has been
considered:
\begin{description}
\item[\AX{(C2)}] If $\lambda(\{v,u\})\ne\emptyset$ for some $u\in\child(v)$,
  then $\lambda(\{v,u'\})=\emptyset$ for \textit{all}
  $u'\in\child(v)\setminus\{u\}$.
\end{description}
This variant imposes an additional condition on the edges $e=uv$ that can
be contracted. More precisely, an inner edge of $(T,t,\lambda)$ can be
contracted without losing the explanation of $(\delta,\varepsilon)$ and
properties \AX{(C1)} and \AX{(C2)} if and only if (i) $t(u)=t(v)$, (ii)
$\lambda(e)=\emptyset$ and (iii) at most one of the the edges $uu'$,
$u'\in\child(u)$ and $vv'$, $v'\in\child(v)$ has a non-empty label. Now
consider two consecutive edges $uv$ and $vw$ with $t(u)=t(v)=t(w)$,
$\lambda(\{u,v\})=\lambda(\{v,w\})=\emptyset$ and suppose there is
$u'\child(u)$ with $\lambda(\{u,u'\})\ne\emptyset$, $w'\in\child(w)$ with
$\lambda(\{w,w'\})\ne\emptyset$, and $\lambda(\{v,v'\})=\emptyset$ for all
$v'\in\child(v)$. Then one can contract either $uv$ or $vw$ but not both
edges. Thus least-resolved trees explaining $(\delta,\varepsilon)$ and
satisfying \AX{(C1)} and \AX{(C2)} are no longer unique.

\section{Concluding Remark}

Here we have shown that symbolic ultrametrics and Fitch maps can be
combined by the simple and easily verified condition that
$\HH(T_{\delta})\cup \HH(T_{\varepsilon})$ is again a hierarchy
(Thm.~\ref{thm:OOXX-tree}), i.e., that the two least-resolved trees have a
common refinement. The least-resolved tree $(T^*,t^*,\lambda^*)$ that
simultaneously explains both $\delta$ and $\varepsilon$ is unique in this
case and can be computed in quadratic time if the label set $N$ is bounded
and $O(|L|^2 |N|)$ time in general. The closely related problem of
combining a hierarchy and \emph{symmetrized Fitch maps}, defined by
$m\in \varepsilon(x,y)$ iff there is an edge $e$ with $m\in\lambda(e)$
along the path from $x$ to $y$ \cite{Hellmuth:18a}, is NP-complete
\cite{Hellmuth:21a}. It appears that the main difference is the fact that
symmetrized Fitch maps do not have a unique least-resolved tree as
explanation. The distinction between much simpler problems in the directed
setting and hard problems in the undirected case is also reminiscent of the
reconciliation problem for trees, which are easy for rooted trees and hard
for unrooted trees, see e.g.\ \cite{Bryant2006}.

We have also seen that certain restrictions on the Fitch maps that are
related to the ``observability'' of horizontal transfer do not alter the
complexity of the problem. These observability conditions are defined in
terms of properties of the explaining trees, raising the question whether
these constraints also have a natural characterization as properties of the
Fitch maps. On a more general level, both symbolic ultrametrics and Fitch
maps arise from evolutionary scenarios comprising an embedding of the gene
tree $T$ into a species tree, with labeling functions $t$ and $\lambda$ on
$T$ encoding event-types and distinctions in the evolutionary fate of
offsprings, respectively. Here we have focused entirely on gene trees with
given labels. The embeddings into species trees are known to impose
additional constraints \cite{Hellmuth:17,LH:20}.

\medskip\noindent
\textbf{Acknowledgments.} 
This work was supported in part by the \emph{Deutsche
  Forschungs\-gemeinschaft}.

\bibliographystyle{abbrv}
\bibliography{OX-preprint}

\end{document}